\documentclass{article}
\usepackage{amsmath}
\usepackage{amsfonts}
\usepackage{amsthm}
\usepackage{latexsym}
\usepackage{amssymb}
\usepackage{verbatim}
\usepackage{enumerate}
\font\logic=msam10 at 10pt

\newcommand{\restrict}{\mbox{\logic\char'026}}

\newtheorem{thrm}{Theorem}[section]

\newtheorem{cor}[thrm]{Corollary}

\newtheoremstyle{hdefinition}%
  {\topsep}%
  {\topsep}%
  {\upshape}
  {}%
  {\bfseries}%
  {.}
  { }%
  {\thmnumber{#2 }\thmname{#1}\thmnote{ \rm(#3)}}%

\newtheoremstyle{hclaim}%
  {\topsep}%
  {\topsep}%
  {\itshape}%
  {}%
  {\bfseries}%
  {.}
  { }%
  {\thmname{#1}\thmnote{ \rm#3}}%

\newtheoremstyle{hnotation}%
  {\topsep}%
  {\topsep}%
  {\upshape}%
  {}%
  {\bfseries}%
  {.}
  { }%
  {\thmname{#1}\thmnote{ \rm#3}}%

\theoremstyle{hclaim}
\newtheorem*{claim*}{Claim}

\theoremstyle{hdefinition}
\newtheorem{df}[thrm]{Definition}

\newtheorem{remark}[thrm]{Remark}

\theoremstyle{hclaim}

\theoremstyle{hnotation}

\begin{document}

\title{The Filter Dichotomy and medial limits}

\author{Paul B. Larson\thanks{
Supported in part by NSF grant DMS-0801009.}}

\maketitle

\begin{abstract} The \emph{Filter Dichotomy} says that every uniform nonmeager
filter on the integers is mapped by a finite-to-one function to an
ultrafilter. The consistency of this principle was proved by Blass and Laflamme.
A \emph{medial limit} is a universally measurable function from
$\mathcal{P}(\omega)$ to the unit interval $[0,1]$ which is finitely additive for
disjoint sets, and maps singletons to $0$ and $\omega$ to $1$.
Christensen and Mokobodzki independently showed that the Continuum Hypothesis implies the existence of medial limits.
We show that the Filter Dichotomy implies
that there are no medial limits.
\end{abstract}

\noindent 2000 Mathematics Subject Classification 03E35; 28A20

\section{Universally measurable sets}

A \emph{measure} on a set $X$ is a function $\mu$ whose domain is some $\sigma$-algebra of subsets of $X$, with codomain $[0, \infty]$, such that $\mu$ is countably additive for disjoint families (by default, measures are countably additive, but we deal also with \emph{finitely additive measures}, which are finitely additive but not necessarily countably additive). A set is said to be \emph{measurable} with respect to $\mu$ if it is in the domain of $\mu$. A \emph{Borel measure} is a measure on a topological space whose domain contains the Borel subsets of the space. A measure is \emph{complete} if all subsets of sets of measure $0$ are in the domain of the measure (and thus have measure 0). The \emph{completion} of a measure is the smallest complete measure extending it.  If $\mu$ is a Borel measure on a topological space $X$, and $\mu^{*}$ is the completion of $\mu$, then a set $A \subseteq X$ is in the domain of $\mu^{*}$ if and only if there is a set $B$ in the domain of $\mu$ such that the symmetric difference $A \bigtriangleup B$ is contained in a set of $\mu$-measure 0 (see 212C of \cite{F2}).
A measure $\mu$ on a set $X$ is a \emph{probability} measure if $\mu(X) = 1$, \emph{finite} if $\mu(X)$ is finite, $\sigma$-\emph{finite} if $X$ is a countable union of sets of finite measure, and \emph{atomless} if singletons have measure 0.

A subset of a topological space is said to be \emph{universally measurable} if it is measurable with respect to every complete $\sigma$-finite
Borel measure on the space (see \cite{Ke, Ni}, and 434D of \cite{F4}, for instance).
The collection of universally measurable sets does not change if one replaces ``$\sigma$-finite" with ``finite"  or requires the measures to be atomless (see 211X(e) of \cite{F2}).

We will make use of the following standard observation.

\begin{remark}\label{immeas} If $f \colon X \to Y$ is a homeomorphism between topological spaces, then any Borel measure $\mu$ on $Y$ induces a measure $\mu^{*}$ on $X$ defined by letting $\mu^{*}(B) = \mu(f[B])$ for any Borel $B \subseteq X$. It follows that, in this context, for any universally measurable $A \subseteq X$, $f[A]$ is a universally measurable subset of $Y$.
\end{remark}



\section{Medial limits}

A function between topological spaces is \emph{universally measurable} if all preimages of open sets (equivalently, Borel sets) are universally measurable.
For our purposes, a \emph{medial limit} is a universally measurable function from $\mathcal{P}(\omega)$ to $[0,1]$ which is finitely additive for
disjoint sets, and maps singletons to $0$ and $\omega$ to $1$ (i.e., a universally measurable finitely additive measure on $\mathcal{P}(\omega)$ giving $\omega$ measure $1$ and singletons measure $0$). Equivalently, a \emph{medial limit} is a function $f \colon \mathcal{P}(\omega) \to [0,1]$, finitely additive for disjoint sets and mapping singletons to $0$ and $\omega$ to $1$, such that
for every complete, $\sigma$-finite Borel measure $\mu$ on $\mathcal{P}(\omega)$ there is a Borel function $b \colon \mathcal{P}(\omega) \to [0,1]$
such that $\{ x \subseteq \omega \mid f(x) \neq b(x)\}$ is $\mu$-null. Medial limits (in various forms) appear the following publications, among others :
\cite{MaPrWe, Ta85, Fi, Ke91, Co, La, JaKeLo, Ka}.
Christensen and
Mokobodzki (see \cite{Ch, Me}) independently showed that medial limits exist under the assumption that the Continuum Hypothesis holds. This assumption was weakened to Martin's Axiom by Normann \cite{No}. As far as we know, the weakest hypothesis known to be sufficient is the statement that the reals are not a union of fewer than continuum many meager sets (i.e., that the covering number for the meager ideal is the continuum). This was apparently known to Mokobodzki in the 1970's (see also 538S of \cite{F5}). The term ``medial limit" is often used for the corresponding linear functional on $\ell^{\infty}$ (see \cite{Me}; 538Q of \cite{F5}).








Godefroy and Talagrand \cite{GoTa} proved in 1977 that if $f$ is a medial limit, then the filter $\{ x\subseteq \omega \mid f(x) = 1\}$ does not have the property of Baire. The next two theorems present their proof of this fact, using terminology from Farah's \cite{Fa}. Note that a filter on $\omega$ has the property of Baire if and only if its corresponding ideal does.

An ideal $I$ on $\omega$ is said to be \emph{c.c.c.\ over Fin} if all almost disjoint families (i.e., families of infinite subsets of $\omega$ which pairwise have finite intersection) disjoint from $I$ are countable (this is Definition 3.3.1 of \cite{Fa}). We say that
an ideal on $\omega$ is \emph{uniform} if it contains all finite subsets of $\omega$, and similarly that a filter is uniform if its corresponding ideal is.

\begin{thrm}\label{medcc} If $f$ is a medial limit and $I = \{ x \subseteq \omega \mid f(x) = 0\}$, then $I$ is a universally measurable uniform c.c.c.\ over Fin ideal.
\end{thrm}

\begin{proof} That $I$ is a universally measurable uniform ideal follows from the definition of medial limit. To see that $I$ is c.c.c.\ over Fin, let $A$ be an almost disjoint family disjoint from $I$, and suppose that $A$ is uncountable. Then there is a positive integer $n$ such that the set of $x \in A$ such that $f(x) \geq 1/n$ is uncountable. Since $f$ is finitely additive for almost disjoint sets, and $f(\omega) = 1$, any such set can have size at most $n$.
\end{proof}

The following is Lemma 3.3.2(c) of \cite{Fa}.

\begin{thrm}\label{ccnb} A uniform c.c.c.\ over Fin proper ideal on $\omega$ cannot have the property of Baire.
\end{thrm}

\begin{proof} Let $I$ be a uniform c.c.c.\ over Fin proper ideal on $\omega$. If $I$ is somewhere comeager, than there are two members of $I$ whose union is cofinite. To see this, suppose that $s\subseteq \omega$ is finite and $D_{n}$ $(n \in \omega)$ are dense open subsets of $\mathcal{P}(\omega)$ such that all $x \in [s] \cap \bigcap_{n \in \omega} D_{n}$ are in $I$, where $[s] = \{ x \subseteq \omega \mid x \cap (max(s) + 1) = s\}$. It is relatively straightforward to build sets $x$, $y$ in $[s] \cap \bigcap_{n \in \omega} D_{n}$ whose union is $\omega\setminus ((max(s) + 1)\setminus s)$. Similarly, if $I$ is meager, then there is a perfect set of subsets of $\omega$ which are almost disjoint and all not in $I$. The construction of such a perfect set is similar, using a collection of dense open sets $D_{n}$ $(n \in \omega)$ such that $\bigcap_{n \in \omega}D_{n}$ is disjoint from $I$.
\end{proof}

\section{The Filter Dichotomy}

\begin{df} The \emph{Filter Dichotomy} is the statement that for each nonmeager filter $F$ on $\omega$, there is a
finite-to-one function $h \colon \omega \to \omega$ such that $\{ h[x] \mid x \in F\}$ is an ultrafilter.
\end{df}

Blass and Laflamme showed \cite{BlLa} that the Filter Dichotomy holds in models previously considered by Miller \cite{Mi} and
Blass and Shelah \cite{BlSh87, BlSh89}.

\begin{thrm}\label{umeasmeager} The Filter Dichotomy implies that universally measurable uniform filters on $\omega$ are meager.
\end{thrm}

\begin{proof} Let $F$ be a nonmeager universally measurable uniform filter on $\omega$, and let $h \colon \omega \to \omega$ be finite-to-one
such that
$\{ h[x] \mid x \in F\}$ is an ultrafilter. Let $$S = \{ \bigcup_{n \in Z}h^{-1}[\{n\}] \mid Z \subseteq \omega\},$$ and
let $G \colon \mathcal{P}(\omega) \to \mathcal{P}(\omega)$ be defined by $G(x) = h[x]$. Then:
\begin{enumerate}
\item $S$ is a perfect subset of $\mathcal{P}(\omega)$.
\item $F \cap S$ is a universally measurable subset of $S$.
\item $G \restrict S \colon S \to \mathcal{P}(\omega)$ is a homeomorphism.
\item $G[F \cap S] = G[F]$ is not Lebesgue measurable.
\end{enumerate}
This gives a contradiction, by Remark \ref{immeas}.

The first and third items above are easy, and the fourth follows from the fact that nonprincipal ultrafilters
are not Lebesgue measurable (\cite{Si}; to see that $G[F]$ has to be nonprincipal, note that $F$ is uniform and $h$ is finite-to-one).
To see the second item above, suppose that $\mu$ is a finite Borel measure on $S$. Define a measure $\mu^{*}$ on
$\mathcal{P}(\omega)$ by letting $\mu^{*}(A) = \mu(A \cap S)$ for all Borel $A \subseteq \mathcal{P}(\omega)$.
Since $F$ is universally measurable, there exist Borel sets $B$ and $N$ such that $F \bigtriangleup B \subseteq N$ and
$\mu^{*}(N) = 0$. Then $\mu(N \cap S) = 0$, and $(F \cap S) \bigtriangleup (B \cap S) \subseteq (N \cap S)$.
\end{proof}

Putting together the Blass-Laflamme result with Lemmas \ref{medcc} and \ref{ccnb} and Theorem \ref{umeasmeager}, we have the following.

\begin{cor}\label{nomed} If ZFC is consistent, then so is ZFC + ``there exist no medial limits."
\end{cor}

Talagrand \cite{Ta80} proved that a filter on $\omega$ is meager if and only if there exists a finite-to-one function $h \colon
\omega \to \omega$ such that $\{ h[x] \mid x \in F\}$ is the set of cofinite subsets of $\omega$. 
The Filter Dichotomy can then be restated
as: for every uniform filter $F$ on $\omega$ there is a finite-to-one function $h \colon \omega \to \omega$ such that $\{ h[x] \mid x \in F\}$ is
either the cofinite filter or a nonprincipal ultrafilter. The condition that $h$ is finite-to-one is used in the proof of Theorem \ref{umeasmeager}
only to get the image ultrafilter to be nonprincipal, and for this one needs only that $h[x]$ is infinite for all $x \in F$. Blass has pointed out to us that the finite-to-one condition can be relaxed in the other (meager) case as well, via the following argument.

\begin{thrm} Suppose that $f \colon \mathcal{P}(\omega) \to [0,1]$ is finitely additive for
disjoint sets, sends singletons to $0$ and $\omega$ to $1$. Let $F = \{ x \subseteq \omega \mid f(x) = 1\}$, and let
$h \colon \omega \to \omega$ be such that $f(h^{-1}[\{n\}]) = 0$ for all $n \in \omega$. Then
$\{ h[x] \mid x \in F\}$ is not the cofinite filter on $\omega$.
\end{thrm}

\begin{proof}  First note that $y \mapsto f(h^{-1}[y])$ defines a function from
$\mathcal{P}(\omega)$ to $[0,1]$ which is finitely additive for
disjoint sets, sends singletons to $0$ and $\omega$ to $1$. Furthermore,
$\{ h[x] \mid x \in F\} = \{ y \subseteq \omega \mid f(h^{-1}[y]) = 1\}$, so it suffices to consider
the case where $h$ is the identity function, and to show that $F$ is not the cofinite filter.
Split $\omega$ into two infinite pieces, and let
$A_{0}$ be the one with smaller measure with respect to $f$, or either piece in case
of a tie. Let $B_{0}$ be the other piece minus its first element. Then $f(A_{0}) \leq 1/2$,  $f(B_{0}) \geq 1/2$
and $| \omega \setminus (A_{0} \cup B_{0})| = 1$ (note that finite sets have $f$-measure $0$).
Split $A_{0}$ into two infinite pieces, let $A_{1}$ be the smaller one with respect to $f$, and let $B_{1}$ be the larger one minus its first element. Then $f(A_{1}) \leq 1/4$, $f(B_{0} \cup B_{1}) \geq 3/4$ and $| \omega \setminus (A_{1} \cup B_{0} \cup B_{1})| = 2$.  Continue in this way, defining infinite sets $A_{n}$ and $B_{n}$ for each $n \in \omega$ such that $f(A_{n}) \leq 2^{-n-1}$, $f(B_{0} \cup \ldots \cup B_{n}) \geq 1 - 2^{-n-1}$
and $|\omega \setminus (A_{n} \cup B_{0} \cup \ldots \cup B_{n})| = n+1$.
Then $\bigcup_{n \in \omega}B_{n}$ has $f$-measure $1$, but is not cofinite.
\end{proof}

From this it follows that the nonexistence of medial limits follows from the following weak form of the Filter Dichotomy:
for every uniform filter $F$ on $\omega$ there exists a function $h \colon \omega \to \omega$ such that $h^{-1}[\omega \setminus \{n\}] \in F$
for all $n \in \omega$ and $\{ h[x] \mid x \in F\}$ is either the cofinite filter or a nonprincipal ultrafilter. Since our proof uses only
the universal measurability of the filter of measure $1$ sets, we have that this weak form of the Filter Dichotomy implies that whenever $f$ is a finitely additive measure on $\mathcal{P}(\omega)$ such that $f(\omega) = 1$, the filter $\{ x \subseteq \omega \mid f(x) = 1\}$ is not universally measurable.

Another proof that the Filter Dichotomy implies the nonexistence of medial limits is given by the logical combination of previously known facts.
We include the following definitions for completeness, the second of which is Definition 538A of \cite{F5}. We refer the reader to \cite{Ke, F5}
for the definition of \emph{probability space}.

\begin{df} Given two filters $F$ and $G$ on $\omega$, $F$ is said to be \emph{Rudin-Keisler} above $G$ ($F \geq_{RK} G$) if there is a
a function $h \colon \omega \to \omega$ such that $\{h[x] \mid x \in F\}$ generates the filter $G$.
\end{df}


\begin{df}
A filter $\mathcal{F}$ on $\omega$ is said to satisfy the \emph{Fatou property} if for any probability space $(X, \Sigma, \mu)$, if $\langle E_{n} : n \in \omega\rangle$
is a sequence in $\Sigma$, and $X = \bigcup_{A \in \mathcal{F}} \bigcap_{n \in A} E_{n}$, then $\lim_{n \to \mathcal{F}} \mu(E_{n})$ is defined and
equal to $1$.
\end{df}

The Filter Dichotomy can be restated as saying that every uniform nonmeager filter on $\omega$ is Rudin-Keisler above a nonprincipal ultrafilter on $\omega$. The nonexistence of medial limits under the Filter Dichotomy then follows from the following facts, where $f$ is a supposed
medial limit and $F$ is the filter $\{ x \subseteq \omega \mid f(x) = 1\}$.

\begin{itemize}
\item The filter $F$ is uniform and nonmeager. (\cite{GoTa})
\item The filter $F$ has the Fatou property. (Proposition 538Rd of \cite{F5})
\item Whenever $G$ and $H$ are filters on $\omega$ such that $G \geq_{RK} H$, and $G$ has the Fatou property, then $H$ has the Fatou property. (Proposition 538Ob of \cite{F5})
\item No nonprincipal ultrafilter on $\omega$ has the Fatou property (Exercise 538Xn of \cite{F5})
\end{itemize}

\section{The Semifilter Trichotomy}

A \emph{semifilter} on $\omega$ is a proper subset of $\mathcal{P}(\omega)$ which is closed under finite changes and supersets. 
The proof of Theorem \ref{umeasmeager} uses only these properties of the filter $F$, and thus gives the following. 

\begin{thrm}\label{threetwosem} If $F$ is a universally measurable semifilter
and $h \colon \omega \to \omega$ is a finite-to-one function, then $\{ h[x] \mid x \in F\}$ is universally measurable. 
\end{thrm} 

The following
statement was shown by Laflamme \cite{Laf} to hold in the models of the Filter Dichotomy mentioned above. 

\begin{df} The \emph{Semifilter Trichotomy} is the statement that for every semifilter $F$ on $\omega$ there is a finite-to-one function 
$h \colon \omega \to \omega$ such that $\{ h[x] \mid x \in F\}$ is either the cofinite filter, a nonprincipal ultrafilter or the set of all infinite
subsets of $\omega$. 
\end{df} 

Talagrand's proof of his result mentioned after Corollary \ref{nomed} applies to semifilters, showing that a semifilter $F$ on $\omega$ is meager if and only if there is a finite-to-one $h \colon \omega \to \omega$ such that $\{h[x] \mid x\in F\}$ is the cofinite filter. Blass has pointed out to us that the third case of the Semifilter Trichotomy occurs if and only if the semifilter $F$ is comeager. We include a proof for the sake of completeness. 

\begin{thrm}\label{blassco} Suppose that $F$ is a semifilter on $\omega$. Then $F$ is comeager if and only if there exists a finite-to-one function 
$h \colon \omega \to \omega$ such that $\{ h[x] \mid x \in F\}$ is the set of all infinite subsets of $\omega$. 
\end{thrm} 

\begin{proof} In each direction we will use the following characterization : a set $Z \subseteq \mathcal{P}(\omega)$ is comeager if and only if there exist a division of $\omega$ into finite sets $I_{k}$ $(k \in \omega)$ and a set $y \subseteq \omega$ intersecting each $I_{k}$ such that every $x \subseteq \omega$ for which $x \cap I_{k} = y \cap I_{k}$ for infinitely many $k$ is in $F$ (see \cite{MaPrWe}, for instance).

For the forward direction of the theorem, applying the assumption that $F$ is comeager we can let $h$ be the function which maps $I_{k}$ to $k$, where the set of $I_{k}$'s is as above. For the reverse direction, we can let $I_{k} = h^{-1}[\{k\}]$ and let $y = \omega$. Then if $x$ is a subset of $\omega$ with the property that $x \cap I_{k} = y \cap I_{k} = I_{k}$ for infinitely many $k$, there exists a $z \in F$ such that $h[z] = \{ k \in \omega \mid
x \cap I_{k} = I_{k}\}$. Then $z \subseteq x$, which implies that $x \in F$. 
\end{proof} 


Putting together Theorems \ref{threetwosem} and \ref{blassco} with Talagrand's characterization of meager semifilters, 
we have the following.

\begin{thrm} The Semifilter Trichotomy implies that universally measurable semifilters on $\omega$ have the property of Baire. 
\end{thrm} 

It is an open question (asked in \cite{MaPrWe}) whether consistently all universally measurable sets of reals have the property of Baire.

\noindent Department of Mathematics\\
Miami University\\
Oxford, Ohio 45056\\United States\\
larsonpb\@@muohio.edu

\end{document}